\documentclass[12pt, twoside, leqno]{article}

\usepackage{amsmath,amsthm}
\usepackage{amssymb}

\usepackage{enumitem}

\usepackage{graphicx}

\usepackage[usenames]{color}
\usepackage[colorlinks,linkcolor=red,anchorcolor=blue,citecolor=blue]{hyperref}

\usepackage[T1]{fontenc}
\newtheorem*{remark1}{Remarks}
\newtheorem{thm}{Theorem}[section]
\newtheorem{prop}[thm]{Proposition}

\newtheorem{lem}[thm]{Lemma}
\newtheorem{cor}[thm]{Corollary}

\newtheorem{ex}[thm]{Example}

\def\N{\mathbb{N}}

\def\N{\mathbb{N}}

\def\PP{\mathcal{P}}

\def\UU{\mathcal{U}}

\def\diam{\text{\rm diam}}

\def\Leb{\text{\rm Leb}}

\def\over{\overline{\rm mdim}_{M}(T,X,d)}

\def\logf{\log\frac{1}{\epsilon}}

\theoremstyle{definition}

\numberwithin{equation}{section}

\begin{document}


\baselineskip=17pt


\title{Measure-theoretic metric mean dimension}

\author{Rui Yang$^{1,2}$\\
E-mail: zkyangrui2015@163.com
\and 
Ercai Chen$^{2}$\\
E-mail: ecchen@njnu.edu.cn
\and  Xiaoyao Zhou$^2$\footnote{corresponding author}\\
E-mail: zhouxiaoyaodeyouxian@126.com
\and  1.School of Mathematics and Statistics,\\ 
Key Laboratory of Nonlinear Analysis and its Applications,\\ Ministry of Education,\\ Chongqing University, Chongqing, 401331, P.R.China\\ 
2.School of Mathematical Sciences and Institute of Mathematics, \\Ministry of Education Key Laboratory of NSLSCS,\\ Nanjing Normal University, Nanjing, 210023, Jiangsu, P.R.China}

\date{}

\maketitle


\renewcommand{\thefootnote}{}



\renewcommand{\thefootnote}{\arabic{footnote}}
\setcounter{footnote}{0}


\begin{abstract}
For infinite measure-theoretic entropy systems, we introduce the notion of measure-theoretic metric mean dimension of invariant measures for different types of measure-theoretic $\epsilon$-entropies,  and show that measure-theoretic metric mean dimensions of  different types of measure-theoretic $\epsilon$-entropies  coincide with the packing metric mean dimension  of  the set of generic points of ergodic measures.
\end{abstract}

\section{Introduction}

~~~~By a pair $(X,T)$ we mean a topological dynamical system (TDS for short), where $X$ is a compact  metric space  with  metric $d$ and  $T:X \rightarrow X$ is a homeomorphism.
By $M(X),M(X,T),E(X,T)$  we denote  the sets of all Borel probability measures on $X$ endowed with the weak$^{*}$-topology, all $T$-invariant Borel probability measures on $X$,  and all ergodic measures on $X$, respectively. 

Mean topological dimension introduced by Gromov \cite{gromov} is a new  topological invariant for topological dynamical systems, which is later proved to  be a powerful tool  to deal with the embedding problems of dynamical systems \cite{lw00, g15,glt16,g17,gt20}.  In 2000, to connect mean dimension  theory and (infinite) entropy theory, Lindenstrauss and Weiss \cite{lw00} introduced the notion of metric mean dimension by  dynamicalizing  the definition of  Minkowski dimension  that is originated from  fractal geometry. This quantities measures how fast $\epsilon$-topological entropy diverges to infinite topological entropy. Although metric mean dimension plays a vital role in  understanding the  complexity of TDSs with  infinite topological entropies, an appropriate notion for infinite measure-theoretic entropy  is absent.  Notice that there are some  
measure-theoretic $\epsilon$-entropies, for instances, Kolmogorov-Sinai $\epsilon$-entropy \cite{w82,p97,gs20}, Brin-Katok $\epsilon$-entropy \cite{bk83},  Katok $\epsilon$-entropy \cite{k80} and  other $\epsilon$-measure entropies \cite{s07,ps07}, being considered to characterize measure theoretic entropy of invariant measures. Namely, the limits of these different measure-theoretic $\epsilon$-entropies are always equal to measure-theoretic entropy of ergodic measures as $\epsilon$ tends $0$. Then a natural question is the following:

{Question:} For   a given TDS with infinite measure-theoretic entropy, do these different measure-theoretic $\epsilon$-entropies of  ergodic measures have the same ``divergence speed'' ?

To this end, in the present paper we introduce the notion of measure-theoretic metric mean dimension of invariant measures for different types of measure-theoretic $\epsilon$-entropies. More precisely, 
given   a TDS  $(X,T)$ with the metric $d$, an invariant measure $\mu \in M(X,T)$, and a measure-theoretic $\epsilon$-entropy  $F(\mu,\epsilon) \in \mathcal{D}$, we  respectively define  \emph{upper and lower measure-theoretic metric mean dimensions of $\mu$} with respect to $F$ as
\begin{align*}
\overline{\rm {mdim}}_M(F,\mu)&=\limsup_{\epsilon \to 0} \frac{F(\mu,\epsilon)}{\logf},\\
\underline{\rm {mdim}}_M(F,\mu)&=\liminf_{\epsilon \to 0} \frac{F(\mu,\epsilon)}{\logf},
\end{align*}
where $F(\mu,\epsilon)$ is chosen from the candidate set
$$\mathcal{D}= \left\{\underline{h}_\mu^K(T,\epsilon,\delta),\overline{h}_\mu^K(T,\epsilon,\delta),\underline{h}_\mu^K(T,\epsilon),  \overline{h}_\mu^K(T,\epsilon), \overline{h}_\mu^{BK}(T,\epsilon),\atop  \mathcal P_{\mu}(T,d,\epsilon), 
PS_\mu(T,\epsilon),  R_{\mu,L^{\infty}}(\epsilon),\inf_{\diam(P) \leq \epsilon}\limits h_\mu(T,P),\inf_{\diam (\UU)\leq \epsilon} \limits h_\mu^S(\mathcal{U}) \right\},$$
and the precise definitions of all  candidates  involved in $\mathcal{D}$  are  given in Subsection \ref{sub 2.2}.

Intuitively, it may happen that different  choice from the candidate set $\mathcal{D}$  shall lead to  different ``divergence speeds''. However,  the following theorem suggests that  the above definitions are well-defined, which do not  depend on  the choice of $F(\mu,\epsilon)$  from a long list of possible candidates, each being a notion of a dynamical ``$\epsilon$-entropy'' of $\mu$. 
\begin{thm}\label{thm 1.1}
Let $(X,T)$ be a TDS  with a metric $d$ and $\mu \in E(X,T)$. Then  for every $F(\mu,\epsilon)\in \mathcal{D}$,
\begin{align*}
\limsup_{\epsilon \to 0}\frac{F(\mu,\epsilon)}{\logf}&=\overline{\rm {mdim}}_M^P(T,G_\mu,d),\\
\liminf_{\epsilon \to 0}\frac{F(\mu,\epsilon)}{\logf}&=\underline{\rm {mdim}}_M^P(T,G_\mu,d),
\end{align*}
$G_\mu:=\{x\in X: \lim\limits_{n\to  \infty}\frac{1}{n}\sum_{j=0}^{n-1}\delta_{T^{j}(x)}=\mu\}$  denotes the set of the generic points of $\mu$, and  $\underline{\rm {mdim}}_M^P(T,G_\mu,d),  \overline{\rm {mdim}}_M^P(T,G_\mu,d)$ denote  packing upper and lower  metric mean dimensions of $G_{\mu},$ respectively. 
\end{thm}

Since  Lindenstrauss-Tsukamoto's  pioneering work \cite{lt18},  the formulation of  variational principles for metric mean dimension  in terms of some  measure-theoretic $\epsilon$-entropies has received  abundant attentions.  Actually, there has been some  variational  principles obtained by  several authors  either  in  terms of invariant measures \cite{lt18,vv17} or in terms of ergodic measures \cite{gs20,shi,ycz23}. By means of the relations of those different types of measure-theoretic $\epsilon$-entropies  presented in Lemmas \ref{lem 3.1} and \ref{lem 3.2}, it is a good chance for us to  give a systemic  summary for these results in   a unified way for future reference. 
\begin{thm}\label{thm 1.2}
Let $(X,T)$ be a TDS  with a metric $d$. Then 
\begin{align*}
\over&=\limsup_{\epsilon \to 0}\frac{1}{\logf}\sup_{\mu \in E(X,T)}F(\mu,\epsilon)\\
&=\limsup_{\epsilon \to 0}\frac{1}{\logf}\sup_{\mu \in M(X,T)}F(\mu,\epsilon),
\end{align*}
where $F(\mu,\epsilon)$ is chosen from the candidate set $$\mathcal{D}_1=\{ R_{\mu,L^{\infty}}(\epsilon),\inf_{\diam(P) \leq \epsilon}\limits h_\mu(T,P)\}.$$
\end{thm}

\begin{thm}\label{thm 1.3}
Let $(X,T)$ be a TDS  with a metric $d$. Then 
\begin{align*}
\over=\limsup_{\epsilon \to 0}\frac{1}{\logf}\sup_{\mu \in M(X)}\underline{h}_\mu^{BK}(T,\epsilon),
\end{align*}
and 
\begin{align*}
\over&=\limsup_{\epsilon \to 0}\frac{1}{\logf}\sup_{\mu \in E(X,T)}F(\mu,\epsilon)\\
&=\limsup_{\epsilon \to 0}\frac{1}{\logf}\sup_{\mu \in M(X,T)}F(\mu,\epsilon)\\
&=\limsup_{\epsilon \to 0}\frac{1}{\logf}\sup_{\mu \in M(X)}F(\mu,\epsilon),
\end{align*}
where $F(\mu,\epsilon)$ is chosen from the candidate set 
$$\mathcal{D}_2= \left\{\underline{h}_\mu^K(T,\epsilon,\delta),\overline{h}_\mu^K(T,\epsilon,\delta),\underline{h}_\mu^K(T,\epsilon),  \overline{h}_\mu^K(T,\epsilon), \atop \overline{h}_\mu^{BK}(T,\epsilon),  \mathcal P_{\mu}(T,d,\epsilon), 
PS_\mu(T,\epsilon),  \inf_{\diam (\UU)\leq \epsilon} \limits h_\mu^S(\mathcal{U},\delta) \right\}.$$
\end{thm}

We  remark that Theorems  \ref{thm 1.2} and \ref{thm 1.3} are also valid for lower metric mean dimension of $X$. Besides, we would like to  mention that  Gutman and  \'Spiewak  in \cite[Problem 3]{gs20} posed a question whether  variational principles hold for  Brin-Katok $\epsilon$-entropy or not. Later, by means of some local variational principles established in local entropy theory,  Shi \cite[Theorem 5.2]{shi}  gave an affirmative  answer for lower Brin-Katok $\epsilon$-entropy $\overline{h}_\mu^{BK}(T,\epsilon)$, and asked   the same question \cite[Problem 1]{shi} for upper Brin-Katok $\epsilon$-entropy  $\underline{h}_\mu^{BK}(T,\epsilon)$.  Theorem  \ref{thm 1.3}  at least tells us the partial answer, which holds  for any  dynamical systems in terms of Borel probability measures. There also exists some certain  infinite entropy systems, for instance, the system $([0.1]^{\mathbb{Z}},\sigma)$ with the left shift  $\sigma$ on the unit cube $[0.1]^{\mathbb{Z}}$, the answer is positive.  However,  how to verify  that the  supremum  of the variational principle can take  the set of all ergodic measures  for any TDS   is still   an open problem.

The rest of this paper is organized as follows. In section 2,  we recall the definitions of metric mean  dimensions  and collect several types of  measure-theoretic  $\epsilon$-entropies. In  section 3, we prove  Theorems  \ref{thm 1.1}, \ref{thm 1.2} and \ref{thm 1.3}. In section 4, we give some applications of Theorem \ref{thm 1.1}. In section 5, we  present some open question suggested by our main results.

\section{Preliminary}
~~~~In this section, we recall the definitions of metric mean  dimensions,  and collect several types of  measure-theoretic  $\epsilon$-entropies related to measure-theoretic entropy of invariant measures.
\subsection{Three types of metric mean dimensions}
We recall the three types of metric mean dimensions defined by $\epsilon$-topological entropy \cite{lw00}, $\epsilon$-Bowen topological entropy \cite{w21} and $\epsilon$-packing topological entropy \cite{ycz22}.
\subsubsection{Metric mean dimension}
~~~~Let  $n\in \mathbb{N}$, $x,y \in X$, $\epsilon >0$. We define the $n$-th Bowen metric $d_n$  on $X$ as $d_n(x,y):=\max_{0\leq j\leq n-1}\limits d(T^{j}(x),T^j(y)).$ Then  {the Bowen open  ball and closed ball} of $x$ with radius $\epsilon$ and  order $n$ in the metric $d_n$ are  given by 
$$B_n(x,\epsilon)=\{y\in X: d_n(x,y)<\epsilon\},$$
$$\overline B_n(x,\epsilon)=\{y\in X:d_n(x,y)\leq\epsilon\},$$ respectively.

For a non-empty subset $Z\subset X$,  a set $F\subset Z$ is  \emph{an $(n,\epsilon)$-separated set of $Z$} if $d_n(x,y)>\epsilon$ for any $x,y \in F$ with $x\not= y$. Denote by $s_n(T,d,\epsilon,Z)$ the maximal  cardinality  of $(n,\epsilon)$-separated sets of $Z$. We  define the $\epsilon$-topological entropy of $X$ as
$$h_{top}(T, Z,d,\epsilon)=\limsup_{n\to \infty} \frac{1}{n} \log s_n(T,d,\epsilon,Z).$$
The \emph{upper metric mean dimension} of  $Z$  is defined by 
\begin{align*}
\overline{\rm mdim}_{M}(T,Z,d)= \limsup_{\epsilon\to 0}\frac{h_{top}(T,Z,d,\epsilon)}{\log \frac{1}{\epsilon}}.
\end{align*}

When  $Z\subset X$ is not  a $T$-invariant  or compact set, motivated by the Bowen topological entropy  \cite{b73} and packing entropy  on subsets \cite{fh12} one can define so-called Bowen  metric mean dimension  \cite{w21,cls21} and packing  metric mean dimension on subsets \cite{ycz22} by means of Carath\'eodory-Pesin structures \cite{p97}. 

\subsubsection{Bowen metric mean dimension}

~~~~Let  $Z\subset X$, $\epsilon>0$, $N\in \mathbb{N}$,  and $s \geq 0$.  
Set
$$M(T,d,Z,s,N,\epsilon)=\inf\sum_{i\in I}\limits  e^{-n_i s},$$
where the infimum  is taken over all  finite or countable covers $\{B_{n_i}(x_i,\epsilon)\}_{i\in I}$ of $Z$ with $n_i \geq N,x_i \in X.$

Since $M(T,d,Z,s,N,\epsilon)$ is non-decreasing when $N$ increases, so the limit $M(T,d,Z,s,\epsilon)=\lim\limits_{N\to \infty}M(T,d,Z,s,N, \epsilon)$ exists. There is a  critical value  of parameter $s$  so that   $M(T,d,Z,s, \epsilon)$ jumps  from $\infty$ to $0$.  The  critical value is defined by 
\begin{align*}
h_{top}^B(T,Z,d,\epsilon):&=\inf\{s:M(T,d,Z,s,\epsilon)=0\}\\
&=\sup\{s:M(T,d,Z,s,\epsilon)=\infty\}.
\end{align*}

We define the  \emph{Bowen upper metric mean dimension of  $Z$} as 
\begin{align*}
\overline{\rm {mdim}}_M^B(T,Z,d)&=\limsup_{\epsilon \to 0}\frac{h_{top}^B(T,Z,d,\epsilon)}{\log \frac{1}{\epsilon}}.
\end{align*}

\subsubsection{Packing metric mean dimension}
~~~~Let  $Z\subset X$, $\epsilon>0$, $N\in \mathbb{N}$, and  $s \geq 0$.  Set
$$P(T,d,Z,s,N,\epsilon)=\sup\sum_{i\in I}\limits  e^{-n_i s},$$
where the supremum is taken over all  finite or countable  pairwise disjoint  closed  ball  families $\{\overline B_{n_i}(x_i,\epsilon)\}_{i\in I}$ of $Z$ with $n_i \geq N, x_i\in Z.$

Since $P(T,d,Z,s,N,\epsilon)$ is non-increasing as $N$ increases, so the limit
$P(T,d,Z,s,\epsilon)=\lim_{N \to \infty }\limits P(T,d,Z,s,N,\epsilon)$ exists.

Set
$$\mathcal P(T,d,Z,s,\epsilon)=\inf\left\{\sum_{i=1}^{\infty}P(T,d,Z_i,s,\epsilon): \cup_{i\geq 1}Z_i \supseteq Z \right\}.$$
There is  a  critical value  of parameter $s$  so that    $\mathcal P(T,d,Z,s,\epsilon)$ jumps  from $\infty$ to $0$.
The  critical value is defined by 
\begin{align*}
h_{top}^P(T,Z,d,\epsilon):
&=\inf\{s: \mathcal P(T,d,Z,s,\epsilon)=0\}\\
&=\sup\{s: \mathcal P(T,d,Z,s,\epsilon)=\infty\}.
\end{align*}

We define \emph{packing upper metric mean dimension  of $T$ on the set $Z$} as
\begin{align*}
\overline{\rm {mdim}}_M^P(T,Z,d)
=\limsup_{\epsilon \to 0}\limits\frac{h_{top}^P(T,Z,d,\epsilon)}{\log \frac{1}{\epsilon}}.
\end{align*}

It is proved that in \cite[Proposition  3.4, (c)]{ycz22} if $Z$ is a non-empty subset of $X$, then
$$\overline{\rm {mdim}}_M^B(T,Z,d)\leq \overline{\rm {mdim}}_M^P(T,Z,d)\leq \overline{\rm {mdim}}_M(T,Z,d).$$
Additionally,   the equalities hold if $Z$ is  a $T$-invariant compact subset of $X$.  

We list some  remarks about metric mean dimension.

\begin{remark1}
\begin{enumerate}[label=\upshape(\roman*), leftmargin=*, widest=iii]
\item  Metric mean dimension  can be  regarded as  a quantity to characterize the topological complexity of  infinite entropy systems, which measures the divergent rate of $\epsilon$-topological entropy as $\epsilon \to 0$.  More precisely, one may think of $\epsilon$-topological entropy 
$$h_{top}(T,X,d,\epsilon) \approx \overline{\rm mdim}_{M}(T,X,d)\cdot \logf$$  for  sufficiently  small $\epsilon >0$, where $\logf$ is the ``standard unit''.\label{it:1}
\item Unlike  the classical topological entropy, metric mean dimension  depends on the choice of compatible metrics of $X$. Some explicit examples can be  constructed  by using the fact  \rm{\cite[Theorem 5]{vv17}} if  $(K,d)$ is a   compact metric space,  then
$\overline{\rm mdim}_M(\sigma, K^{\mathbb{Z}},d)=\overline{{\rm dim}}_B(K,d),$
where  $\sigma$ is the left shift on $(K^{\mathbb{Z}},d)$ with a metric $\tilde{d}(x,y)=\sum_{n\in \mathbb{Z}}\frac{d(x_n,y_n)}{2^{|n|}}$ compatible with the product topology, and  $ \overline{{\rm dim}}_B(K,d)$  is the upper box dimension of $K$ w.r.t.  $d$.  
Consider  $X=\{0,1\}^{\mathbb{{N}}}$ equipped with product topology, which is compatible with the metrics
$$d_1(x,y)=\sum_{n\in \mathbb{N}}\frac{|x_n-y_n|}{2^{n}}, d_2(x,y)=\sum_{n\in \mathbb{N}}\frac{|x_n-y_n|}{3^{n}}.$$

Then by using  the result that we mention, one has
\begin{align*}
\overline{\rm mdim}_M(\sigma, X^{\mathbb{Z}},\tilde{d}_1)&=\overline{{\rm dim}}_B(X,d_1)=1,\\
\overline{\rm mdim}_M(\sigma, X^{\mathbb{Z}},\tilde{d}_2)&=\overline{{\rm dim}}_B(X,d_2)=\frac{\log2}{\log3},\\
\overline{\rm mdim}_M(\sigma, X^{\mathbb{Z}},\tilde{d}_1)&\not= \overline{\rm mdim}_M(\sigma, X^{\mathbb{Z}},\tilde{d}_2), 
\end{align*} 
where $\tilde{d}_1(x,y)=\sum_{n\in \mathbb{Z}}\frac{d_1(x_n,y_n)}{2^{|n|}}$, $\tilde{d}_2(x,y)=\sum_{n\in \mathbb{Z}}\frac{d_2(x_n,y_n)}{2^{|n|}}$. Let $id: (\sigma, X^{\mathbb{Z}},\tilde{d}_2) \rightarrow (\sigma, X^{\mathbb{Z}},\tilde{d}_1)$ be the identity mapping. This is exactly  a conjugation mapping.   It confirms that metric mean dimension  varies  with the metrics  and is not  a topological invariant.

\item Suppose that $X,Y$ are two compact metrizable spaces. Let $\pi: X\rightarrow Y$ be a conjugation mapping between  two TDSs $(X,T)$ and $(Y,S)$,  where $Y$ is endowed with a compatible metric $d$. Then   one has
$$\overline{\rm {mdim}}_M(T,X,\widehat{d})=\overline{\rm {mdim}}_M(S,Y,d),$$
where  $\widehat{d}(x,y):=d(\pi x, \pi y)$,  defined for any $x,y\in X$, is compatible  with the topology of $X$. In other words, for any  conjugation mapping between two TDSs,  once we are given  one of a TDS with  a  compatible metric,  one can  find another compatible metric for its extension system (or factor system) such that they share the same metric mean dimensions.

\item  If  $\pi:(X, T, d_1) \mapsto (Y,S,d_2)$ is a bi-Lipschitz conjugation map,  then $\overline{\rm {mdim}}_M(T,X,{d}_1)=\overline{\rm {mdim}}_M(S,Y,d_2)$. Therefore,    bi-Lipschitz  conjugation mapping  between two TDSs preserves their  metric mean  dimensions. Specially,  the metric mean dimensions of $X$ do not change whenever the  metric is taken over the set of all   compatible  metrics such that  the identity mapping $id$ on $X$ is a  bi-Lipschitz  mapping. 
\end{enumerate}
\end{remark1}

\subsection{Measure-theoretic $\epsilon$-entropies} \label{sub 2.2}

~~~~In this subsection,  we collect some  candidates for measure-theoretic $\epsilon$-entropy, which were used to describe  the measure-theoretic entropy of invariant  measures.

\subsubsection{Kolmogorov-Sinai $\epsilon$-entropy}

~~~Let $P$ be a  finite Borel measurable partition of $X$. By  $\diam (P):=\max_{A\in P}\diam A$ and $ P^n:=\vee_{j=0}^{n-1}T^{-j}P$  we denote  the diameter  of  $P$,  the $n$-th join of $P$, respectively.
Given an invariant measure $\mu \in M(X,T)$,  the measure-theoretic entropy   of $\mu$ w.r.t $P$ is defined  by 
$$h_{\mu}(T,P)=\lim\limits_{n \to \infty}\frac{1}{n}H_{\mu}(P^n),$$
where $H_{\mu}(P)$ is the  partition entropy of $P$.  The measure-theoretic entropy of $\mu$ is given by 
$$h_{\mu}(T)=\sup h_{\mu}(T,P),$$
where the supremum is taken over all finite Borel measurable partitions of $X$.

The  \emph{Kolmogorov-Sinai $\epsilon$-entropy of $\mu$} \cite[Section 3]{gs20} is  given by 
$$\inf_{\diam(P) \leq \epsilon} h_\mu(T,P),$$ 
where  the infimum is   taken over all finite Borel partitions of $X$ with diameter  at most $\epsilon$.

\subsubsection{Shapira's $\epsilon$-entropy}
~~~~Given  a finite open cover  $\mathcal{U}$ of $X$, we denote by $\diam(\mathcal{U})=\max_{U\in \mathcal{U}}\diam U$  the diameter of $\mathcal{U}$.
The \emph{Lebesgue number} of $\UU$, denoted by $\Leb(\mathcal{U})$,  is  the largest  positive number $\delta >0$ such that  every open ball of radius $\delta$ of $X$ is contained in  some element of $\UU$. If every atom of the finite Borel partition  $P$ of $X$ is  contained in some member of $\UU$,  in this case we write  $P\succ \UU$  to denote that  $P$  is  a refinement $\UU$.  Let $ \mathcal{U}^n:=\vee_{j=0}^{n-1}T^{-j}\mathcal{U}$  denote the $n$-th join of $\mathcal{U}$.

Given  $\mu \in M(X)$, $\delta \in (0,1)$ and  a  finite open cover $\mathcal{U}$  of $X$, we set
$$N_{\mu}(\mathcal{U},\delta):=\min\{\#\mathcal{F}:\mu(\cup_{A\in \mathcal{F}}A)\geq \delta,  \mathcal{F}~\text{is a subfamily of}~ \mathcal{U}\}.$$ 
We define the  Shapira's  $\epsilon$-entropy of $\mu$ as
$$h_\mu^S(\mathcal{U},\delta):=\limsup_{n\to \infty}\frac{\log N_{\mu}(\mathcal{U}^n,\delta)}{n}.$$

Shapira  \cite[Theorem 4.2]{s07}  proved  that  for every ergodic measure   $\mu \in E(X,T)$,   the limit
exists and is independent of  choice of $\delta \in (0,1)$, which we denote by $h_\mu^S(\mathcal{U})$ instead of $h_\mu^S(\mathcal{U},\delta)$.  Moreover, Shapira \cite[Theorem 4.4]{s07}  showed that $$h_\mu^S(\mathcal{U})=\inf_{P\succ \mathcal{U}}h_\mu(T,P),$$ 
where   the infimum  is taken over all  finite Borel measurable partitions  of $X$  that refines $\mathcal{U}$.
This  links  the Shapira's entropy with  the  measure-theoretic entropy of a fixed finite open cover.


\subsubsection{Brin-Katok $\epsilon$-entropy} 

~~~~Let  $\mu \in {M}(X)$ and $\epsilon >0$. We define  \emph{the upper and lower Brin-Katok local $\epsilon$-entropies of $\mu$} as
\begin{align*}
\overline{h}_{\mu}^{BK}(T, \epsilon):&=\int \limsup_{n\to \infty}-\frac{\log \mu (B_n(x,\epsilon))}{n}d\mu,\\
\underline{h}_{\mu}^{BK}(T, \epsilon):&=\int \liminf_{n\to \infty}-\frac{\log \mu (B_n(x,\epsilon))}{n}d\mu,
\end{align*}
respectively.

Brin and Katok \cite{k80} proved that for every $\mu \in M(X,T)$, one has $\lim_{\epsilon\to 0}\underline{h}_{\mu}^{BK}(T, \epsilon)=\lim_{\epsilon\to 0}\overline{h}_{\mu}^{BK}(T, \epsilon)=h_\mu(T).$

\subsubsection{Katok's $\epsilon$-entropies}
~~~~Let $\mu \in  {M}(X)$, $\epsilon>0$, $n \in\N$ and $\delta \in (0,1)$.
Put
$$R_\mu^\delta(T,n, \epsilon):=\min\{\#E: E\subset X  ~\text{and} ~\mu (\cup_{x\in E}B_n(x,\epsilon))\geq 1-\delta \}.$$
We define the \emph{upper and lower Katok's $\epsilon$-entropies  of $\mu$} as
\begin{align*}
\overline{h}_{\mu}^K(T,\epsilon, \delta)&=\limsup_{n\to \infty} \frac{1}{n} \log R_\mu^\delta(T,n, \epsilon),\\
\underline{h}_{\mu}^K(T,\epsilon, \delta)&=\liminf_{n\to \infty} \frac{1}{n} \log R_\mu^\delta(T,n, \epsilon),
\end{align*}
respectively.

Notice that  the quantities $\overline{h}_{\mu}^K(T,\epsilon, \delta),\underline{h}_{\mu}^K(T,\epsilon, \delta)$ are  increasing as $\delta$ decreases. An alternative  way to define the \emph{upper and lower Katok's $\epsilon$-entropies  of $\mu$} is  the following:
\begin{align*}
	\overline{h}_{\mu}^K(T,\epsilon):=\lim_{\delta \to 0}\overline{h}_{\mu}^K(T, \epsilon, \delta),~~
	\underline{h}_{\mu}^K(T,\epsilon):=\lim_{\delta \to 0}\underline{h}_{\mu}^K(T,\epsilon, \delta).
\end{align*}
If $\mu \in E(X,T)$, Katok \cite{k80} showed  that  for every $\delta \in(0,1)$, one has
$\lim_{\epsilon \to 0}\overline{h}_{\mu}^K(T,\epsilon, \delta)=\lim_{\epsilon \to 0}\underline{h}_{\mu}^K(T,\epsilon, \delta)=h_{\mu}(T).$

Similar to  the definitions of Bowen and packing topological entropies on  subsets,  one can employ Carath\'eodory-Pesin structures to define  the \emph{dimensional types of  Katok entropy-like $\epsilon$-entropies}.

Let $\epsilon>0, s \geq 0, N\in \mathbb{N}, \mu \in M(X)$ and $\delta \in (0,1)$. 
Put
$$M^\delta (T,d,\mu,s,N,\epsilon)=\inf\{\sum_{i\in I}\limits  e^{-n_i s}\},$$
where the infimum  is taken over all  finite or countable covers $\{B_{n_i}(x_i,\epsilon)\}_{i\in I}$ so  that  $\mu (\cup_{i\in I}B_{n_i}(x_i,\epsilon))\geq 1-\delta$  with $n_i \geq N,  x_i\in X$.

We define $M^\delta (T,d,\mu,s,\epsilon)=\lim\limits_{N\to \infty}M^\delta(T,d,\mu,s,N, \epsilon).$ There is   a critical value  of parameter $s$  so that  the quantity  $M^\delta(T,d,\mu,s, \epsilon)$ jumps  from $\infty$ to $0$.  We define the critical value as
\begin{align*}
M^\delta(T,d,\mu,\epsilon):&=\inf\{s:M^\delta(T,d,\mu,s,\epsilon)=0\}\\
&=\sup\{s:M^\delta(T,d,\mu,s,\epsilon)=\infty\}.
\end{align*}

We put $M_{\mu}(T, d, \epsilon)=\lim_{\delta \to 0 }\limits M^\delta(T,d,\mu,\epsilon)$.

 Set
$$\mathcal P^\delta(T,d,\mu,s,\epsilon)=\inf\left\{\sum_{i=1}^{\infty}P(T,d, Z_i, s,\epsilon): \mu(\cup_{i\geq 1}Z_i)\geq 1-\delta \right\}.$$
There is a  critical value  of parameter $s$  so that  the quantity  $\mathcal P^\delta(T,d,\mu,s,\epsilon)$ jumps  from $\infty$ to $0$. 
We  define the critical value as
\begin{align*}
\mathcal P^\delta(T,d,\mu,\epsilon):
&=\inf\{s: \mathcal P^\delta(T,d,\mu,s,\epsilon)=0\}\\
&=\sup\{s: \mathcal P^\delta(T,d,\mu,s,\epsilon)=\infty\}.
\end{align*}

We put  $\mathcal P_{\mu}(T,d,\epsilon)=\lim_{\delta \to 0}\limits \mathcal P^\delta(T,d,\mu,\epsilon)$.

\subsubsection{Pfister and Sullivan's  $\epsilon$-entropy}

~~~~Let $\mu \in M(X)$ and $\epsilon>0$.
Define 
$$PS_\mu(T,d,\epsilon)=\inf_{F\ni \mu}\limsup_{n \to \infty}\frac{1}{n}\log s_n(T,d,\epsilon,X_{n,F}),$$
where $X_{n,F}=\{x\in X: \frac{1}{n}\sum_{j=1}^{n-1} \delta_{T^{j}(x)} \in F\}$ and  the infimum  is taken over  all neighborhoods  $F$  in $M(X)$ of $\mu$.

Pfister and Sullivan  \cite{ps07} proved that $h_{\mu}(T)=\lim_{\epsilon \to 0}\limits PS_\mu(T,d,\epsilon)$ for every $\mu \in E(X,T)$.

\subsubsection{$L^{\infty}$ rate-distortion function}\label{sub 2.2.5}
~~~~ This quantity comes  from information theory.
Let $\mu \in M(X,T)$. Recall that  \emph{upper  and lower $L^{\infty}$ rate-distortion dimensions of $\mu$, defined by $L^{\infty}$ rate-distortion function $R_{\mu,L^{\infty}}(T, \epsilon)$}, are respectively given by 
\begin{align*}
\overline{\rm rdim}_{L^{\infty}}(T,X,d,\mu)&=\limsup_{\epsilon \to 0}\frac{R_{\mu,L^{\infty}}(T, \epsilon)}{\log \frac{1}{\epsilon}},\\
\underline{\rm rdim}_{L^{\infty}}(T,X,d,\mu)&=\liminf_{\epsilon \to 0}\frac{R_{\mu,L^{\infty}}(T,\epsilon)}{\log \frac{1}{\epsilon}}.
\end{align*}

Since the forthcoming  proof does not refer to the  definition of $R_{\mu,L^{\infty}}(T,\epsilon)$, so  we  omit its  lengthy  definition and  refer readers to \cite{ct06, lt18} for the details.

\section{Proofs of  main results}

~~~~This section is devoted to proving  the Theorems \ref{thm 1.1}, \ref{thm 1.2} and \ref{thm 1.3}.

According to  the definition of measure-theoretic metric mean dimension, we need to find out the relationships  of these different types of measure-theoretic $\epsilon$-entropies, so the Bowen's \cite{b73}  techniques involving the measure-theoretic entropy do not work in current situation. The  Shannon-McMillan-Breiman theorem, Brin-Katok entropy formula, Shapira's theorems and dimensional approach  are the main ingredients for paving  the way for the proof  of Theorem 1.1. Then   the Theorem 1.1  directly  follows from the following Lemmas \ref{lem 3.1} and  \ref{lem 3.2}.
\begin{lem}\label{lem 3.1}
Let $(X,T)$ be a TDS  with a metric $d$ and $\mu \in E(X,T)$. Then for every $\delta \in (0,1)$ and $\epsilon>0$, we have
\begin{align*}
&\overline{h}_\mu^{BK}(T,2\epsilon)\leq\inf_{\diam(P) \leq \epsilon}h_\mu(T,P)\leq\inf_{\diam (\UU)\leq \epsilon}h_\mu^S(\mathcal{U})
\leq \underline{h}_\mu^{K}(T,\frac{\epsilon}{4},\delta)\\
\leq& \overline{h}_\mu^{K}(T,\frac{\epsilon}{4},\delta)
\leq \underline{h}_\mu^{K}(T,\frac{\epsilon}{32})
\leq 
\overline{h}_\mu^{K}(T,\frac{\epsilon}{32})\leq \overline{h}_\mu^{BK}(T,\frac{\epsilon}{64}).
\end{align*}
\end{lem}

\begin{proof}
Fix $\epsilon >0$ and $\delta \in (0,1)$. Let $P$ be a finite Borel measurable partition of $X$ with diameter at most $\epsilon$, and $P^n(x)$  denote the atom of the partition $P^n=\vee_{j=0}^{n-1}T^{-j}P$  containing $x$.  Then for any $n\in \mathbb{N}$ and $x\in X$ one has $P^n(x)\subset B_n(x,2\epsilon)$. Applying  Shannon-McMillan-Breiman theorem  and  noticing that $ P$  is arbitrary, it yields that
$$\overline{h}_\mu^{BK}(T,2\epsilon)\leq \inf_{\diam(P) \leq \epsilon}\limits h_\mu(T,P).$$ 

Given a finite open cover $\UU $ of $X$ with diameter at most $\epsilon$ and $\gamma >0$,  by \cite[Theorem 4.4]{s07} $h_\mu^S(\mathcal{U})=\inf_{P\succ \mathcal{U}}h_\mu(T,P)$,  there exists a finite  Borel  measurable partition  $P\succ\mathcal{U}$ such that  $\inf_{\diam(Q) \leq \epsilon}\limits h_\mu(T,Q)\leq h_\mu(T,P)<h_\mu^S(\mathcal{U})+\gamma$, which implies that  $$
\inf_{\diam(P) \leq \epsilon}\limits h_\mu(T,P) \leq \inf_{\diam(\UU)\leq \epsilon}\limits h_\mu^S(\mathcal{U}).$$

Choose a finite open cover  $\mathcal{ U}$ of $X$  with $\diam(\mathcal{U})\leq \epsilon$ and $\Leb(\mathcal{U})\geq \frac{\epsilon}{4}$(such an open cover exists by considering a finite $\frac{\epsilon}{4}$-net of $X$, cf. \cite[Lemma 3.4]{gs20}). Since every Bowen ball $B_n(x,\frac{\epsilon}{4})$ is contained in some member of $\mathcal{U}^n$, then
$ N_\mu (\mathcal{ U}^n,1-\delta) \leq  R_\mu ^{\delta}(T,n,\frac{\epsilon}{4})$.  Therefore,  by  \cite[Theorem 4.2]{s07}, one has 
\begin{align*}
\inf_{\diam (\UU)\leq \epsilon}\limits h_\mu^S(\mathcal{U}) \leq h_\mu^S(\mathcal{U})=\lim_{n\to \infty}\frac{\log N_{\mu}(\mathcal{U}^n,	1-\delta)}{n}\leq \underline{h}_\mu^{K}(T,\frac{\epsilon}{4},\delta).
\end{align*}

Again, let $\mathcal{V}$ be a finite open cover of $X$  with $\diam(\mathcal{ V})\leq \frac{\epsilon}{8}(<\frac{\epsilon}{4})$ and $\Leb(\mathcal{V})\geq \frac{\epsilon}{32}$. 
Then  each member  of $\mathcal{V}^n$  is contained in  $B_n(x,\frac{\epsilon}{4})$ for some $x \in X$, and  each  Bowen ball $B_n(x,\frac{\epsilon}{32})$ is contained in some members of $\mathcal{V}^n$. This gives us  $ R_\mu ^{\delta}(T,n,\frac{\epsilon}{4})\leq N_\mu (\mathcal{ V}^n,1-\delta) \leq  R_\mu ^{\delta}(T,n,\frac{\epsilon}{32})$.  Then
\begin{align*}
\overline{h}_\mu^{K}(T,\frac{\epsilon}{4},\delta)
\leq h_\mu^S(\mathcal{V})&=\lim_{n\to \infty}\frac{\log N_{\mu}(\mathcal{V}^n,1-\delta)}{n},~\text{by  \cite[Theorem 4.2]{s07}} \\
&= \liminf_{n\to \infty}\frac{\log N_{\mu}(\mathcal{V}^n,1-\delta)}{n}, ~\text{by  \cite[Theorem 4.2]{s07}}\text{~again} \\
&\leq  \underline{h}_\mu^{K}(T,\frac{\epsilon}{32},\delta)\\
&\leq  \underline{h}_\mu^{K}(T,\frac{\epsilon}{32}).
\end{align*}
The last inequality $\overline{h}_\mu^{K}(T,\frac{\epsilon}{32})\leq \overline{h}_\mu^{BK}(T,\frac{\epsilon}{64})$ follows from \cite[Proposition 3.21]{ycz22}.
\end{proof}

\begin{lem}\label{lem 3.2}
Let $(X,T)$  be a TDS  with a metric $d$ and  $\mu \in E(X,T)$.  Then for any $\epsilon >0$, 
\begin{align*}
\overline{h}_\mu^{BK}(T,\epsilon)&\leq \PP_{\mu}(T,d,\frac{\epsilon}{10})\leq \inf\{h_{top}^P(T,Z,d,\frac{\epsilon}{10}): \mu(Z)=1\}
\leq 
h_{top}^P(T,G_\mu,d,\frac{\epsilon}{10})\\
&\leq PS_\mu(T,\frac{\epsilon}{10})\leq R_{\mu,L^{\infty}}(T,\frac{\epsilon}{60})
\leq \overline{h}_\mu^K(T,\frac{\epsilon}{60})\leq \overline{h}_\mu^{BK}(T,\frac{\epsilon}{120}).
\end{align*}
\end{lem}

Using  the facts $PS_\mu(T,\epsilon)\leq R_{\mu,L^{\infty}}(T,\frac{\epsilon}{6})\leq \overline{h}_\mu^K(T,\frac{\epsilon}{6})$  \cite[Proposition 4.3]{w21}  and  
$ \overline{h}_\mu^K(T,2\epsilon)\leq \overline{h}_\mu^{BK}(T,\epsilon)$ \cite[Proposition 3.21]{ycz22},  and  noting that $\epsilon>0$  is arbitrary, this shows 
$$PS_\mu(T,\frac{\epsilon}{10})\leq R_{\mu,L^{\infty}}(T,\frac{\epsilon}{60})
\leq \overline{h}_\mu^K(T,\frac{\epsilon}{60})\leq \overline{h}_\mu^{BK}(T,\frac{\epsilon}{120}).$$
The following two  propositions are  to show the remaining inequalities given in Lemma \ref{lem 3.2}: 
\begin{align*}
\overline{h}_\mu^{BK}(T,\epsilon)\leq \PP_{\mu}(T,d,\frac{\epsilon}{10})\leq
\inf\{h_{top}^P(T,Z,d,\frac{\epsilon}{10}): \mu(Z)=1\}\\
\leq h_{top}^P(T,G_\mu,d,\frac{\epsilon}{10})\leq  PS_\mu(T,\frac{\epsilon}{10}).
\end{align*}

\begin{prop}
Let $(X,T)$  be a TDS  with a metric $d$ and  $\mu \in E(X,T)$.  Then for any $\epsilon >0$, 
$$h_{top}^P(T,G_\mu,d,\epsilon) \leq PS_\mu(T,\epsilon).$$
\end{prop}

\begin{proof}
Fix  a neighborhood $F$ of $\mu$. Let $$G_{\mu,F}^N:=\{x\in G_\mu: \frac{1}{n}\sum_{j=1}^{n-1} \delta_{T^{j}(x)}\in F, \forall n\geq N\}$$ and fix $N_0\in\N$. Then $G_\mu =\cup_{N\geq 1}G_{\mu,F}^N$ and  $G_{\mu,F}^{N_0} \subset X_{n,F}$ for any $n\geq N_0$. Without loss  of generality, we assume that  $h_{top}^P(T,G_{\mu,F}^{N_0},d,\epsilon)>0$. Let $0<s<t< h_{top}^P(T,G_{\mu,F}^{N_0},d,\epsilon)$. Then  
$P(T,d,G_{\mu,F}^{N_0},t,\epsilon)\geq \mathcal P(T,d,G_{\mu,F}^{N_0},t,\epsilon) =\infty$. Hence,
for any $N\geq N_0$ we  can choose a  finite or countable  pairwise disjoint  closed  ball family $\{\overline B_{n_i}(x_i,\epsilon)\}_{i\in I}$  with $n_i \geq N, x_i\in G_{\mu,F}^{N_0}$
satisfying
$\sum_{i\in I}e^{-n_it}>1.$

Put $I_k:=\{i\in I: n_i=k\}$, $k\geq N$. Then  
$\sum_{k\geq N}\sum_{i\in I_k}e^{-kt}>1$ and the set  $\{x_i:i\in I_{k}\}$ is a $({k},\epsilon)$-separated set of $X_{{k},F}$. There exists $k\geq N$ (depending on $N$) such that $\#I_{k}>(1-e^{s-t})e^{sk}$. This implies that $s_{{k}}(T,d,\epsilon,X_{{k},F})\geq (1-e^{s-t})e^{sk}$, which yields that 
$\limsup_{n \to \infty}\limits \frac{\log s_n(T, d, \epsilon, X_{n,F})}{n}\geq s.$
Letting $s\to  h_{top}^P(T,G_{\mu,F}^{N_0},d,\epsilon)$, we get 
$$\limsup_{n \to \infty}\limits \frac{\log s_n(T,d,\epsilon,X_{n,F})}{n}\geq h_{top}^P(T,G_{\mu,F}^{N_0},d,\epsilon).$$

Since  $G_\mu =\cup_{N_0\geq 1}G_{\mu,F}^{N_0}$,   by the definition of $\mathcal P(T,d,G_{\mu},\lambda,\epsilon)$, one has
$\mathcal{P}(T,d,G_\mu,\lambda,\epsilon)\leq \sum_{N_0\geq 1}\mathcal{P}(T,d,G_{\mu,F}^{N_0},\lambda,\epsilon)$ for any $\lambda\geq0$. This implies that for $\epsilon>0$, $$h_{top}^P(T,G_{\mu},d,\epsilon)\leq  \sup_{N_0\geq 1}h_{top}^P(T,G_{\mu,F}^{N_0},d,\epsilon)\leq h_{top}^P(T,G_{\mu},d,\epsilon).$$
Using this fact  and  noting that $F$ is arbitrary, one has
$ h_{top}^P(T,G_{\mu},d,\epsilon) \leq PS_\mu(T,\epsilon).$
\end{proof}

\begin{prop}\label{prop 3.4}
Let $(X,T)$ be a TDS  with a metric $d$  and $\mu \in M(X)$. Then for every $\epsilon >0$,
\begin{align*}
\underline{h}_\mu^{BK}(T,\epsilon)\leq M_{\mu}(T,d,\frac{\epsilon}{2})&\leq \inf\{h_{top}^B(T,Z,d,\frac{\epsilon}{2}): \mu(Z)=1\},\\
\overline{h}_\mu^{BK}(T,\epsilon)\leq \PP_{\mu}(T,d,\frac{\epsilon}{10})&\leq \inf\{h_{top}^P(T,Z,d,\frac{\epsilon}{10}): \mu(Z)=1\}.
\end{align*}
\end{prop}

\begin{proof}
Fix $\epsilon >0$.  We firstly  show  inequality
$$\underline{h}_\mu^{BK}(T,\epsilon)\leq M_{\mu}(T,d,\frac{\epsilon}{2})\leq \inf\{h_{top}^B(T,Z,d,\frac{\epsilon}{2}): \mu(Z)=1\}.$$
Obviously, the inequality  $M_{\mu}(T,d,\frac{\epsilon}{2})\leq \inf\{h_{top}^B(T,Z,d,\frac{\epsilon}{2}): \mu(Z)=1\}$ holds by definitions.  Assume that $\underline{h}_\mu^{BK}(T,\epsilon)>0$ and let $0<s<\underline{h}_\mu^{BK}(T,\epsilon)$. By a standard method, there exist a Borel set $ A$ with $\mu(A)>0$ and $N_0$ such that 
$\mu(B_n(x,\epsilon))<e^{-ns}$
for any $x\in A$ and $n\geq N_0$.
Let $\delta =\frac{1}{2}\mu(A)>0$ and fix an integer  number $N\geq N_0$.  Let  $\{B_{n_i}(x_i,\frac{\epsilon}{2})\}_{i \in I}$  be a finite or countable family so that  $\mu(\cup_{i\in I}B_{n_i}(x_i,\frac{\epsilon}{2}))\geq 1-\delta$ with $x_i\in X$ and $n_i\geq N$.  Then  $\mu(A\cap \cup_{i\in I}B_{n_i}(x_i,\frac{\epsilon}{2}))\geq \frac{1}{2}\mu(A)>0$. 
Put $$I_1=\{i \in I: A\cap B_{n_i}(x_i,\frac{\epsilon}{2})\not=\emptyset\}.$$
For each $i\in I_1$ we choose $y_i\in  A\cap B_{n_i}(x_i,\frac{\epsilon}{2})$ such that $  B_{n_i}(y_i,\epsilon)\supset A\cap B_{n_i}(x_i,\frac{\epsilon}{2})$. Notice  that
\begin{align*}
\sum_{i\in I}e^{-sn_i}\geq \sum_{i\in I_1}e^{-sn_i}\geq \sum_{i\in I_1}\mu(B_{n_i}(y_i,\epsilon))\geq\frac{1}{2}\mu(A).
\end{align*}
Then  $M^\delta (T,d,\mu,s,\frac{\epsilon}{2})\geq M^\delta (T,d,\mu,s,N,\frac{\epsilon}{2})>0$ and hence $M_{\mu}(T,d,\frac{\epsilon}{2})\geq M^\delta (T,d,\mu,\frac{\epsilon}{2})\geq s$. Letting $s \to \underline{h}_\mu^{BK}(T,\epsilon) $ we obtain that
$\underline{h}_\mu^{BK}(T,\epsilon)\leq M_{\mu}(T,d,\frac{\epsilon}{2})$.

Next, we  show
$$\overline{h}_\mu^{BK}(T,\epsilon)\leq \PP_{\mu}(T,d,\frac{\epsilon}{10})\leq \inf\{h_{top}^P(T,Z,d,\frac{\epsilon}{10}): \mu(Z)=1\}.$$ 
Since $\PP_{\mu}(T,d,\frac{\epsilon}{10})\leq \inf\{h_{top}^P(T,Z,d,\frac{\epsilon}{10}): \mu(Z)=1\}$ by definitions, so  it suffices to show the inequality $\overline{h}_\mu^{BK}(T,\epsilon)\leq \PP_{\mu}(T,d,\frac{\epsilon}{10})$. Let $0<s<\overline{h}_\mu ^{BK}(T,\epsilon)$. Then one can choose  $\theta >0$ and a Borel set $A$ with $\mu(A)>0$ so that 

$$\limsup_{n \to \infty }-\frac{\log \mu(B_n(x,\epsilon))}{n}>s+\theta$$
for all $x\in A$. Let  $\delta \in (0,\mu(A))$ and  $\{Z_i\}_{i\geq1}$  be  a collection  of  subsets of $X$ with  $\mu(\cup_{i\in 1}Z_i)\geq 1-\delta$. Then $\mu(A\cap Z_i)>0$ for some $i\geq 1$. Fix such $i$ and put
$$E_n:=\left\{x\in A\cap Z_i: \mu(B_n(x,\epsilon))<e^{-(s+\theta)n}\right\}.$$ Then  $A\cap Z_i=\cup_{n\geq N}E_n$ for every $N\in\mathbb{N}$. Fix  an integer number $N\geq 1$. There exists  $n\geq N$ so that 
$\mu(E_n)\geq \frac{1}{n(n+1)}\mu(A\cap Z_i).$ Fix such  $n$. We consider a closed  cover  $\{\overline{B}_n(x,\frac{\epsilon}{10}):x\in E_n\}$ of $E_n$. By 5r-covering lemma \cite[Theorem 2.1]{m95},  there exists  a finite  pairwise disjoint subfamily 
$\{\overline{B}_n(x_i,\frac{\epsilon}{10}):i\in I\}$, where $I$ is a  finite index set, so  that 
$$\cup_{i\in I}B_n(x_i,\epsilon)\supseteq\cup_{i\in I}\overline{B}_n(x_i,\frac{\epsilon}{2})\supseteq \cup_{x\in E_n}\overline{B}_n(x,\frac{\epsilon}{10}).$$
Hence, 
\begin{align*}
P(T,d,A\cap Z_i,N,s,\frac{\epsilon}{10})&\geq
P(T,d,E_n,N,s,\frac{\epsilon}{10})
\geq \sum_{i\in I} e^{-ns}\\
&\geq e^{n\theta} \sum_{i\in I} \mu(B_n(x_i,\epsilon))\geq e^{n\theta} \mu (E_n)\\
&\geq e^{n\theta}\frac{\mu(A\cap Z_i)}{n(n+1)}.
\end{align*}
It follows that $P(T,d,A\cap Z_i,s,\frac{\epsilon}{10})=\infty $ by  letting $n \to \infty$ and hence  $\sum_{i\geq 1}P(T,d, Z_i,s,\frac{\epsilon}{10})=\infty$. Therefore,  one has $\mathcal P_{\mu}(T,d,\frac{\epsilon}{10})\geq \mathcal P^\delta(T,d,\mu,\frac{\epsilon}{10})$\\$\geq s.$
Letting $s \to \overline{h}_\mu ^{BK}(T,\epsilon)$ gives us  the desired result.
\end{proof}

A natural follow-up question to Theorem \ref{thm 1.1} is \emph{whether it is possible to include the  lower Brin-Katok $\epsilon$-entropy  $\underline{h}_{\mu}^{BK}(T,\epsilon)$  on the list $\mathcal{D}$?}
	
Until  now, we do not have a proper answer for this question except computing  the  precise   ``divergence speed'' of $ \underline{h}_{\mu}^{BK}(T,\epsilon)$ as $\epsilon \to 0$. Before that, we  invoke  a  3r-covering lemma developed in \cite[Lemma 1]{mw08}.
\begin{lem}\label{lem 3.5}
Let $r>0$ and $\mathcal{B}(r)=\{B_n(x,r):x\in X, n\geq 1\}$. For each family $\mathcal{F}\subset \mathcal{B}(r)$, there exists a subfamily $\mathcal{G}\subset \mathcal{F}$ consisting  of disjoint balls  such that
$$\cup_{B\in \mathcal{F}}B\subset \cup_{ B_n(x,r)\in \mathcal{G}}B_n(x,3r).$$
\end{lem}

\begin{prop}\label{prop 3.6}
Let $(X,T)$ be a TDS with a metric $d$. If $\mu \in E(X,T)$, then  for every $\epsilon>0$,  there exists a Borel  set $E_{\epsilon}$ with  $\mu$-full measure so that
\begin{align*}
h_{top}^B(T,E_{\epsilon}\cap G_{\mu},d,3\epsilon)\leq  \underline{h}_{\mu}^{BK}(T,\epsilon).
\end{align*}
\end{prop}

\begin{proof}
Fix  $\epsilon >0$ and  set  $$E_{\epsilon}=\{x\in X:\underline{h}_{\mu}^{BK}(T, x,\epsilon)=\underline{h}_{\mu}^{BK}(T,\epsilon)\},$$
where  $\underline{h}_{\mu}^{BK}(T, x,\epsilon)=\liminf_{n\to \infty}-\frac{\log \mu (B_n(x,\epsilon))}{n}$. Notice that  $\mu$ is $T$-invariant and  $T(B_{n+1}(x,\epsilon))\subset B_n(Tx,\epsilon)$ for each $n \geq 1$. One has 
$$\mu(B_{n+1}(x,\epsilon))= \mu(T(B_{n+1}(x,\epsilon)))\leq  \mu(B_n(Tx,\epsilon)),$$
which implies  that  $\underline{h}_{\mu}^{BK}(T, Tx,\epsilon) \leq \underline{h}_{\mu}^{BK}(T,^{} x,\epsilon)$  for all  $x\in X$. Since $\int  \underline{h}_{\mu}^{BK}(T, x,\epsilon)-\underline{h}_{\mu}^{BK}(T, Tx,\epsilon) d\mu=0$, we have $\underline{h}_{\mu}^{BK}(T, Tx,\epsilon)=\underline{h}_{\mu}^{BK}(T, x,\epsilon)$ for a.e. $x\in X$.  The ergodicity of $\mu$ yields that $\underline{h}_{\mu}^{BK}(T, x,\epsilon)$  is a constant and exactly equals to  $\underline{h}_{\mu}^{BK}(T,\epsilon)$ for $\mu$-a.e. $x\in X$.
Therefore,  one has $\mu(E_{\epsilon}\cap G_{\mu})=1$.

Let $s>\underline{h}_{\mu}^{BK}(T,\epsilon)$.   For  each $x\in E_{\epsilon} \cap G_{\mu}$,  we can choose a  strictly  increasing sequence $n_j(x)$ that converges to $\infty$ so that $\mu(B_{{n_j}(x)}(x,\epsilon))>e^{-n_j(x)s}$. Given  $N\geq 1$,  the set $E_{\epsilon} \cap G_{\mu}$ is contained  in the family $$\mathcal{F}_N:= \{x\in E_{\epsilon} \cap G_{\mu}:\mu(B_{{n_j}(x)}(x,\epsilon))>e^{-n_j(x)s},n_j(x)\geq N\}.$$
By Lemma \ref{lem 3.5},  there  exists  a  subfamily  of pairwise disjoint Bowen balls $\{B_{n_i}(x_i,\epsilon)\}_{i\in I} \subset \mathcal{F}_N$ such that
$$E_{\epsilon} \cap G_{\mu}\subset \cup_{i\in I}B_{{n_i}}(x_i,3\epsilon),$$
where  the cardinality of $I$ is  at most countable  since  $\mu(X)=1$  and    each  Bowen open ball  $B_{{n_i}}(x_i,\epsilon)$  has positive  $\mu$-measure.
It follows that
\begin{align*}
M(T,d,E_{\epsilon} \cap G_{\mu},s,N,3\epsilon)\leq \sum_{i\in I}e^{-n_is}\leq  \sum_{i\in I}\mu(B_{{n_i}}(x_i,\epsilon))\leq 1,
\end{align*}
which implies that  $h_{top}^B(T,E_{\epsilon}\cap G_{\mu},d,3\epsilon)\leq s$. Letting $s \to \underline{h}_{\mu}^{BK}(T,\epsilon)$ gives us 
$h_{top}^B(T,E_{\epsilon}\cap G_{\mu},d,3\epsilon)\leq  \underline{h}_{\mu}^{BK}(T,\epsilon) $.
\end{proof}

This following Theorem gives the precise  divergent rate of lower Brin-Katok $\epsilon$-entropy of  an ergodic measure.

\begin{thm}\label{thm 3.7}
Let $(X,T)$ be a TDS  with a metric $d$ and $\mu \in E(X,T)$. Then  there exists a  Borel subset $Z\subset X$ with  $\mu$-full measure  such that

\begin{enumerate}[label=\upshape(\roman*), leftmargin=*, widest=iii]
	\item    $\liminf_{\epsilon \to 0}\frac{\underline{h}_{\mu}^{BK}(T,\epsilon)}{\logf}= \liminf_{\epsilon \to 0}\frac{M_{\mu}(T,d,\epsilon)}{\logf}=\underline{\rm {mdim}}_M^B(T,G_\mu\cap Z,d)$;
	\item  $\limsup_{\epsilon \to 0}\frac{\underline{h}_{\mu}^{BK}(T,\epsilon)}{\logf}= \limsup_{\epsilon \to 0}\frac{M_{\mu}(T,d,\epsilon)}{\logf}=\limsup_{k \to \infty} \frac{h_{top}^B(T, G_\mu\cap Z, d, \epsilon_k)}{\log \frac{1}{\epsilon_k}}$ for some subsequence ${\epsilon_k}$ of positive real numbers that converges to $0$ as $k \to \infty$.
\end{enumerate}
\end{thm}
\begin{proof}
 Take  the set $E_\epsilon$ as in Proposition \ref{prop 3.6}.  Then by  Propositions \ref{prop 3.4} and \ref{prop 3.6},  for every $\epsilon >0$, one has
$$\underline{h}_\mu^{BK}(T,6\epsilon)\leq M_{\mu}(T,d,3\epsilon)\leq h_{top}^B(T,E_{\epsilon}\cap G_{\mu},d,3\epsilon)\leq  \underline{h}_{\mu}^{BK}(T,\epsilon).$$
This shows $\liminf_{\epsilon \to 0}\frac{\underline{h}_{\mu}^{BK}(T,\epsilon)}{\logf}= \liminf_{\epsilon \to 0}\frac{M_{\mu}(T,d,\epsilon)}{\logf}$  and
$\limsup_{\epsilon \to 0}\frac{\underline{h}_{\mu}^{BK}(T,\epsilon)}{\logf}= \limsup_{\epsilon \to 0}\frac{M_{\mu}(T,d,\epsilon)}{\logf}.$

Choose two sub-sequences $\{\epsilon_j\}_{j\geq 1}$, $\{\epsilon_k\}_{k\geq 1}$ of positive real numbers that converge to $0$ as $j,k \to \infty$  so that $\lim_{j \to \infty}\frac{\underline{h}_{\mu}^{BK}(T,\epsilon_j)}{\log\frac{1}{\epsilon_j}}=\liminf_{\epsilon \to 0}\frac{\underline{h}_{\mu}^{BK}(T,\epsilon)}{\logf}$  and  $\lim_{k \to \infty}\frac{\underline{h}_{\mu}^{BK}(T,2\epsilon_k)}{\log\frac{1}{2\epsilon_k}}=\limsup_{\epsilon \to 0}\frac{\underline{h}_{\mu}^{BK}(T,\epsilon)}{\logf}$. Let $$Z=(\cap_{j\geq 1}E_{\epsilon_j}) \cap (\cap_{k\geq 1}E_{\epsilon_k}),$$
which has  $\mu$-full measure.   Using Proposition \ref{prop 3.4} again,  for every $\epsilon >0$ we have
$$\underline{h}_\mu^{BK}(T,\epsilon)\leq  h_{top}^B(T,Z\cap G_\mu,d,\frac{\epsilon}{2}).$$
On the other hand,  by Proposition \ref{prop 3.6} one has
$$h_{top}^B(T,Z\cap G_{\mu},d,3\epsilon_j)\leq h_{top}^B(T,E_{\epsilon_j}\cap G_{\mu},d,3\epsilon_j)\leq  \underline{h}_{\mu}^{BK}(T,\epsilon_j)$$
for every $j$. This shows  $\underline{\rm {mdim}}_M^B(T,Z\cap G_\mu,d)= \liminf_{\epsilon \to 0}\frac{\underline{h}_{\mu}^{BK}(T,\epsilon)}{\logf}$.  
The similar arguments give us 
\begin{align*}
\limsup_{\epsilon \to 0}\frac{\underline{h}_{\mu}^{BK}(T,\epsilon)}{\logf}&=\limsup_{k \to \infty}\frac{\underline{h}_{\mu}^{BK}(T,2\epsilon_k)}{\log\frac{1}{2\epsilon_k}}\leq 
\limsup_{k \to \infty}{ \frac{h_{top}^B(T,Z\cap G_{\mu},d,\epsilon_k)}{\log \frac{1}{\epsilon_k}}}\\
&\leq \limsup_{k \to \infty}{ \frac{h_{top}^B(T,E_{\epsilon_k}\cap G_{\mu},d,\epsilon_k)}{\log \frac{1}{\epsilon_k}}}\\
&\leq   \limsup_{k \to \infty}\frac{\underline{h}_{\mu}^{BK}(T,\frac{\epsilon_k}{3})}{\log\frac{1}{\epsilon_k}}
\leq \limsup_{\epsilon \to 0}\frac{\underline{h}_{\mu}^{BK}(T,\epsilon)}{\logf}.
\end{align*}
\end{proof}

Now, we are in a position  to give the proofs of Theorems  \ref{thm 1.2} and \ref{thm 1.3}.

\begin{proof}[Proof of Theorems  \ref{thm 1.2} and  \ref{thm 1.3}]

In \cite[Theorem 3.1 and Remark 3.6]{gs20},
Gutman and  \'Spiewak showed  that 
\begin{align*}
\over&=\limsup_{\epsilon \to 0}\frac{1}{\logf}\sup_{\mu \in E(X,T)}\inf_{\diam(P) \leq \epsilon}\limits h_\mu(T,P)\\
&=\limsup_{\epsilon \to 0}\frac{1}{\logf}\sup_{\mu \in M(X,T)}\inf_{\diam(P) \leq \epsilon}\limits h_\mu(T,P).
\end{align*}
By virtue of  Lemmas \ref{lem 3.1} and \ref{lem 3.2},  one has  the  ergodic variational principles for metric mean dimension:
$$\over=\limsup_{\epsilon \to 0}\frac{1}{\logf}\sup_{\mu \in E(X,T)}F(\mu,\epsilon)$$
for every $F(\mu,\epsilon)\in \mathcal{D}$. In particular, if $F(\mu,\epsilon)=R_{\mu,L^{\infty}}(\epsilon)$,
then
\begin{align*}
\over&=\limsup_{\epsilon \to 0}\frac{1}{\logf}\sup_{\mu \in E(X,T)}R_{\mu,L^{\infty}}(\epsilon)\\
&\leq\limsup_{\epsilon \to 0}\frac{1}{\logf}\sup_{\mu \in M(X,T)}R_{\mu,L^{\infty}}(\epsilon)\\
&\leq\over,
\end{align*}
where the last inequality  follows by  comparing the definitions of metric mean dimension and rate-distortion function (cf.  \cite[Lemma 34]{lt18}). This shows the  Theorem \ref{thm 1.2}.

By using geometric measure theory and  dimensional approach,  the authors \cite[Theorem 1.4]{w21}, \cite[Theorem 1.4]{ycz22} established variational principles  for Bowen  and packing metric mean dimensions of compact subsets of $X$.  Precisely, 
\begin{align*}
\overline{\rm {mdim}}_M^B(T,K,d)&=\limsup_{\epsilon \to 0}\frac{1}{\logf}\sup_{\mu \in M_K(X)}\underline{h}_\mu^{BK}(T,\epsilon),\\
\overline{\rm {mdim}}_M^P(T,K,d)&=\limsup_{\epsilon \to 0}\frac{1}{\logf}\sup_{\mu \in M_K(X)}\overline{h}_\mu^{BK}(T,\epsilon),
\end{align*}
where  $M_K(X):=\{\mu \in M(X),\mu(K)=1\}$. Then the Proposition \ref{prop 3.4} gives us another variational principles on compact subsets: 
\begin{align*}
\overline{\rm {mdim}}_M^B(T,K,d)&=\limsup_{\epsilon \to 0}\frac{1}{\logf}\sup_{\mu \in M_K(X)}M_{\mu}(T, d, \epsilon),\\
\overline{\rm {mdim}}_M^P(T,K,d)&=\limsup_{\epsilon \to 0}\frac{1}{\logf}\sup_{\mu \in M_K(X)}\mathcal{P}_{\mu}(T,d,\epsilon).
\end{align*}
Together with the fact $$\over=\overline{\rm {mdim}}_M^B(T,X,d)=\overline{\rm {mdim}}_M^P(T,X,d)$$ given in \cite[Proposition 3.4, (c)]{ycz22}, we have  the following  variational principles 
\begin{align*}
\over&=\limsup_{\epsilon \to 0}\frac{1}{\logf}\sup_{\mu \in M(X)}F(\mu,\epsilon)
\end{align*}
for every  $F(\mu,\epsilon)\in\{\underline{h}_\mu^{BK}(T,\epsilon), \overline{h}_\mu^{BK}(T,\epsilon),\mathcal{P}_{\mu}(T,d,\epsilon)\}$. 

To get   the remaining equalities of Theorem \ref{thm 1.3}, according to the ergodic variational principles for $F(\mu,\epsilon)$ we only need to show 
$$\limsup_{\epsilon \to 0}\frac{1}{\logf}\sup_{\mu \in M(X)}F(\mu,\epsilon)\leq\over ,$$
for every $$F(\mu,\epsilon)\in\left\{ \underline{h}_\mu^K(T,\epsilon,\delta), \overline{h}_\mu^K(T,\epsilon,\delta),\underline{h}_\mu^K(T,\epsilon),\atop \overline{h}_\mu^K(T,\epsilon),PS_\mu(T,d,\epsilon) 
,\inf_{\diam (\UU)\leq \epsilon} \limits h_\mu^S(\mathcal{U,\delta})\right\}.$$
Given $\mu \in M(X)$ and $\epsilon >0$, if  $$F(\mu,\epsilon)\in\{ \underline{h}_\mu^K(T,\epsilon,\delta), \overline{h}_\mu^K(T,\epsilon,\delta),\underline{h}_\mu^K(T,\epsilon),  \overline{h}_\mu^K(T,\epsilon),PS_\mu(T,d,\epsilon) 
\},$$
then  $F(\mu,\epsilon)\leq h_{top}(T,X,d,\epsilon)$ by definitions; if $F(\mu,\epsilon)=\inf_{\diam (\UU)\leq \epsilon} \limits h_\mu^S(\mathcal{U},\delta)$, analogous to the proof of the Lemma \ref{lem 3.1} one can show $F(\mu,\epsilon)\leq \underline{h}_\mu^K(T,\frac{\epsilon}{4},\delta)\leq h_{top}(T,X,d,\frac{\epsilon}{4})$
by considering a finite open cover  $\mathcal{ U}$ of $X$ with $\diam(\mathcal{U})\leq \epsilon$ and $\Leb(\mathcal{U})\geq \frac{\epsilon}{4}$. This completes the proof.
\end{proof}

\section{Applications of  main results}
~~~~In this section, we provide some applications of main results.

Katok  entropy formula states that for every $\delta \in(0,1)$,
$\lim_{\epsilon \to 0}\overline{h}_{\mu}^K(T,\epsilon, \delta)=\lim_{\epsilon \to 0}\underline{h}_{\mu}^K(T,\epsilon, \delta)=h_{\mu}(T)$. In infinite entropy theory,  by Theorem \ref{thm 1.1} such an analogue is the following:
\begin{cor}
Let $(X,T)$ be a TDS with a metric $d$. Suppose that  $\mu \in E(X,T)$. Then for every  $\delta \in (0,1)$ and $F(\mu,\epsilon)\in\{\underline{h}_\mu^K(T,\epsilon,\delta),\overline{h}_\mu^K(T,\epsilon,\delta)\}$, 
\begin{align*}
\limsup_{\epsilon \to 0}\frac{F(\mu,\epsilon)}{\logf}&=\limsup_{\epsilon \to 0}\frac{\inf_{\diam(P) \leq \epsilon} h_\mu(T,P)}{\logf},\\
\liminf_{\epsilon \to 0}\frac{F(\mu,\epsilon)}{\logf}&=\liminf_{\epsilon \to 0}\frac{\inf_{\diam(P) \leq \epsilon} h_\mu(T,P)}{\logf}.
\end{align*} 
\end{cor}

To connect ergodic theory and  metric mean dimension theory, Lindenstrauss and Tsukamoto \cite{lt18} introduced the concept of $L^{\infty}$ rate-distortion function,  and  proved  variational principles for metric mean dimension in terms of  $L^{\infty}$ rate-distortion function. The following corollary offers another bridges between information theory and entropy theory of dynamical systems.
\begin{cor}
Let $(X,T)$ be a TDS with a metric $d$. Then for  every $\mu \in E(X,T)$, one has
\begin{align*}
h_{top}^P(T,G_\mu)=h_\mu(T)&=\lim_{\epsilon \to 0 }R_{\mu,L^{\infty}}(T,\epsilon)=\sup_{\epsilon >0}R_{\mu,L^{\infty}}(T,\epsilon),\\
h_{top}(T,X)&=\sup_{\mu \in E(X,T)}\lim_{\epsilon \to 0 }R_{\mu,L^{\infty}}(T,\epsilon),\\
\overline{\rm {mdim}}_M^P(T,G_\mu,d)&=\overline{\rm rdim}_{L^{\infty}}(T,X,d,\mu),\\
\underline{\rm {mdim}}_M^P(T,G_\mu,d)&=\underline{\rm rdim}_{L^{\infty}}(T,X,d,\mu),
\end{align*}
where $h_{top}^P(T,G_\mu)=\lim\limits_{\epsilon \to 0}h_{top}^P(T,G_{\mu},d,\epsilon)$ denotes the packing topological entropy of $G_{\mu}$.
\end{cor}
\begin{proof}
The first  equality follows by Lemma  \ref{lem 3.2} and Brin-Katok entropy formula in ergodic case \cite{bk83}. The second one is due to the classical variational principle for topological entropy \cite{w82}. The last two equalities are a direct consequence of Theorem  \ref{thm 1.1}.
\end{proof}

Ornstein and Weiss \cite{ow93} proved that  the measure-theoretic entropy of an ergodic  measure $\mu$ with respect to the finite  Borel measurable  partition $Q$ can be equivalently given by 
$$h_{\mu}(T,Q)=\lim\limits_{n\to \infty}\frac{\log R_n(x,Q)}{n}$$
for $\mu$-a.e. $x \in X$, where  $R_n(x,Q)=\inf\{j\geq 1:T^jx\in Q_n(x)\}$ is the first  return time of $x$ with respect to $Q_n$, and $Q_n(x)$ is the atom of  the partition $\vee_{j=0}^{n-1}T^{-j}Q$  to which $x$ belongs.

It turns out that the measure-theoretic entropy of $\mu$ equals to  the supremum of the exponential growth rates of Poincar\'e recurrences over all  finite Borel measurable partitions of $X$. See \cite{ow93,var09,dw04} for more details about  the links between measure-theoretic entropy  and Poincar\'e recurrence. The  next corollary suggests that packing metric mean dimension  can be also  expressed as Poincar\'e recurrences. 
\begin{cor}
Let $(X,T)$ be a TDS with a metric $d$. Then for  every $\mu \in E(X,T)$,
\begin{align*}
\overline{\rm {mdim}}_M^P(T,G_\mu,d)&= \limsup_{\epsilon \to 0}\frac{1}{\logf}\inf_{\diam (Q) \leq \epsilon}\limits\int \lim\limits_{n\to \infty}\frac{\log R_n(x,Q)}{n}d\mu,\\
\underline{\rm {mdim}}_M^P(T,G_\mu,d)&= \liminf_{\epsilon \to 0}\frac{1}{\logf}\inf_{\diam (Q) \leq \epsilon}\limits\int \lim\limits_{n\to \infty}\frac{\log R_n(x,Q)}{n}d\mu.
\end{align*}
\end{cor}
Given an ergodic measure  $\mu \in E(X,T)$, in \cite{b73,p97} the authors established  an equality  stating that  measure-theoretic entropy of $\mu$ is the infimum of the Bowen topological entropy of Borel sets of $X$ with $\mu$-full measure: $$h_\mu(T)=\inf\{h_{top}^{B}(T,Z): \mu(Z)=1\},$$ where $h_{top}^{B}(T,Z)$ denotes the Bowen topological entropy of $Z$. This equality is known as  \emph{inverse  variational principle} of Bowen topological entropy by Pesin \cite{p97} in his monograph.  Readers can turn to \cite{p97,chz13,cls21} for more results of this aspect.  The following  inverse  variational principles   exhibit some new connections between ergodic theory and metric mean dimension theory. 

\begin{cor}
Let $(X,T)$ be a TDS with a metric $d$. Suppose that  $\mu \in E(X,T)$. Then  for every $F(\mu,\epsilon)\in \mathcal{D}$, 
\begin{align*}
\overline{\rm {mdim}}_M(F,\mu)
&=\min\{\overline{\rm {mdim}}_M^P(T,Z,d): \mu(Z)=1\},\\
\underline{\rm {mdim}}_M(F,\mu)
&=\min\{\underline{\rm {mdim}}_M^P(T,Z,d): \mu(Z)=1\}.
\end{align*}

\end{cor}

\begin{proof}
By Lemmas \ref{lem 3.1} and  \ref{lem 3.2},   $\overline{\rm {mdim}}_M(F,\mu)
\leq\overline{\rm {mdim}}_M^P(T,Z,d)$ for each Borel set $Z$ with $\mu(Z)=1$, and the  set   $G_{\mu}$ is the one that realizes the infimum.
\end{proof}
As we all know,  the equalities $$h_{top}^B(T,G_\mu)=h_{top}^P(T,G_\mu)=h_\mu (T)$$  hold for every  $\mu \in E(X,T)$ \cite{b73,dzz20}. Unfortunately, we only get such equalities for Bowen and packing metric mean dimensions  under an additional condition.

\begin{cor}\label{cro 3.6}
Let $(X,T)$ be a TDS with a metric $d$. For  every $\mu \in E(X,T)$ and $F(\mu,\epsilon)\in \mathcal{D}$,
\begin{align*}
\limsup_{\epsilon \to 0}\frac{\underline{h}_\mu^{BK}(T,\epsilon)}{\logf}&=\limsup_{\epsilon \to 0}\frac{M_{\mu}(T,d,\epsilon)}{\logf}\leq 
\overline{\rm {mdim}}_M^B(T,G_\mu,d)\\
&\leq \overline{\rm {mdim}}_M^P(T,G_\mu,d)=\limsup_{\epsilon \to 0}\frac{F(\mu,\epsilon)}{\logf}.
\end{align*}
Consequently, if 
$\mu $  satisfies that
$\limsup_{\epsilon \to 0}\frac{\overline{h}_\mu ^{BK}(T,\epsilon)}{\logf}=\limsup_{\epsilon \to 0}\frac{\underline{h}_\mu ^{BK}(T,\epsilon)}{\logf},$
then  for every $F(\mu,\epsilon)\in \mathcal{D}\cup\left\{\underline{h}_\mu^{BK}(T,\epsilon),M_{\mu}(T,d,\epsilon)\right\}$, 
\begin{align*}
\overline{\rm {mdim}}_M^B(T,G_\mu,d)
=\overline{\rm {mdim}}_M^P(T,G_\mu,d)
=\limsup_{\epsilon \to 0}\frac{F(\mu,\epsilon)}{\logf}.
\end{align*}
\end{cor}
\begin{proof}
The inequality $\limsup_{\epsilon \to 0}\frac{\underline{h}_\mu^{BK}(T,\epsilon)}{\logf}=\limsup_{\epsilon \to 0}\frac{M_{\mu}(T,d,\epsilon)}{\logf}\leq 
\overline{\rm {mdim}}_M^B(T,G_\mu,d)$ is guaranteed by Theorem \ref{thm 3.7}.   The inequality $\overline{\rm {mdim}}_M^B(T,G_\mu,d) \leq \overline{\rm {mdim}}_M^P(T,G_\mu,d)$ dues to  \cite[Proposition 3.4]{ycz22}. The last equality follows by Theorem \ref{thm 1.1}.
\end{proof}

\begin{ex}
Let $\sigma:[0,1]^{\mathbb{Z}}\rightarrow [0,1]^{\mathbb{Z}}$ be the left shift on alphabet $[0,1]$, where $[0,1]$ is the unit interval with the standard metric, and let $d(x,y)=\sum_{n\in \mathbb{Z}}2^{-|n|}|x_n-y_n|$ be the metric defined on 
$[0,1]^{\mathbb{Z}}$. Endow $[0,1]^{\mathbb{Z}}$ with product measure $\mu=\mathcal{L}^{\otimes \mathbb{Z}}$, where $\mathcal{L}$  is  the  Lebesgue measure  on $[0,1]$. 

The authors showed that \rm{\cite[Example 3.24]{ycz22}}
$$\limsup_{\epsilon \to 0}\limits  \frac{\overline h_{\mu}^{BK}(\sigma,\epsilon)}{\log \frac{1}{\epsilon}}=\limsup_{\epsilon \to 0} \limits \frac{\underline h_{\mu}^{BK}(\sigma,\epsilon)}{\log \frac{1}{\epsilon}}=\overline{\rm {mdim}}_M^B(\sigma,G_{\mu},d)=1.$$ It is well-known \cite{lt18} that  $\overline{\rm {mdim}}_M(\sigma,[0,1]^{\mathbb{Z}},d)=1$. By \cite[Proposition 3.4, (c)]{ycz22}, we have  $$1=\overline{\rm {mdim}}_M^B(\sigma,G_{\mu},d)\leq\overline{\rm {mdim}}_M^P(\sigma,G_{\mu},d)\leq \overline{\rm {mdim}}_M(\sigma,[0,1]^{\mathbb{Z}},d)=1.$$
This shows that
$$\limsup_{\epsilon \to 0}\limits  \frac{\overline h_{\mu}^{BK}(\sigma,\epsilon)}{\log \frac{1}{\epsilon}}=\limsup_{\epsilon \to 0} \limits \frac{\underline h_{\mu}^{BK}(\sigma,\epsilon)}{\log \frac{1}{\epsilon}}=\overline{\rm {mdim}}_M^P(\sigma,G_{\mu},d)=1.$$

\end{ex}

\section{Open questions} 
~~~~Finally, we  end up this paper by  giving  some open questions arising  from the present paper and  existing  for  a  long time in metric mean dimension theory, which enjoys some special attentions by many researchers in this  aspect. We believe that  more efforts toward these   questions shall   give  us a good chance of understanding the  geometric  and topological  structures of infinite entropy systems.
\begin{enumerate}
\item [(1).] Given a TDS  $(X,T)$, does 
$$\limsup_{\epsilon \to 0}\frac{\underline{h}_\mu^{BK}(T,\epsilon)}{\logf}= \limsup_{\epsilon \to 0}\frac{\overline{h}_\mu^{BK}(T,\epsilon)}{\logf}$$
hold for every $\mu \in E(X,T)$?
\item [(2).] Do there exist  $c>0$ and  $F(\mu,\epsilon)\in \mathcal{D}$ such that $F(\mu,c\epsilon)\leq \underline{h}_\mu^{BK}(T,\epsilon)$ for every $\epsilon >0$ and $\mu \in E(X,T)$?
\item [(3).] Under which condition on $(X,T)$, is there a $F(\mu,\epsilon)\in \mathcal{D}$ so that
$$\over=\sup_{\mu \in E(X,T)}\limsup_{\epsilon \to 0}\frac{F(\mu,\epsilon)}{\logf}?$$
\end{enumerate}

We remark that  Question 1 is an analogue of  classical  Brin-Katok entropy formula.  Question 2 implies  Question 1 by Theorem \ref{thm 1.1}, and  allows us   to obtain ergodic  variational principles for metric mean dimension in terms of $\underline{h}_\mu^{BK}(T,\epsilon)$ by Theorems \ref{thm 1.2} and  \ref{thm 1.3}. Hence, the central question is how to solve Question 2. There is a possible way of modifying the  Bowen's work  \cite[Lemma 1]{b73}  and using Proposition \ref{prop 3.6} to obtain   $\inf_{\diam (P)\leq c\epsilon}\limits h_{\mu}(T,P)\leq \underline{h}_\mu^{BK}(T,\epsilon)$ for some constant $c>0$, but   so far after several attempts the authors  still  failed to  get   the  inequality. 

The  hardcore of Question 3 is whether  we can  exchange the  order of $\limsup_{\epsilon \to 0}$ and $\sup_{\mu \in E(X,T)}$ within variational principles of metric mean dimensions in terms of some measure-theoretic $\epsilon$-entropies. 
We summarize the progress  of this aspect.

\begin{itemize}
\item [(i)]   Lindenstrauss and Tsukamoto firstly posed counter-examples (see \cite[Section VIII]{lt18})  showing that the variational  principles are not valid for rate-distortion  functions if we try to exchange the  order of $\limsup_{\epsilon \to 0}$ and $\sup_{\mu \in M(X,T)}$. This  at the same time  stimulates several authors  asking analogous question  for Katok's  $\epsilon$-entropies in their consequent work \cite{vv17, cl23}.
\item [(ii)] If  $F(\mu,\epsilon)=\inf_{\diam P \leq \epsilon}\limits h_{\mu}(T,P)$, for systems with marker property, for instance,   aperiodic  minimal systems and  aperiodic finite-dimensional  systems(cf. \cite[Theorem 6.1]{g15}), the authors \cite[Theorem 1.3]{ycz23} showed that, for Theorem \ref{thm 1.2}, the order of   $\limsup_{\epsilon \to 0}$ and $\sup_{\mu \in M(X,T)}$  can be exchanged for some ``nice''  compatible metrics.
\item  [(iii)]  Inspired by the work  of Ruelle  \cite{rue72} and Walters \cite{w82}, by using convex analysis  approach the authors   \cite[Theorem 1.1]{ycz22b}   firstly  formulated  a variational principle for upper metric mean dimension with potential in terms of upper measure-theoretic metric mean dimension (cf. \cite[Definition 3.4]{ycz22b} for its precise definition), whose form is close to the classical variational principle of topological entropy \cite{w82}. This result was further extended to  random dynamical systems \cite[Theorem 1.1]{ycz22c},  scaled pressure functions \cite[Theorems A and B]{cpv24} and other generalized framework within metric mean dimension theory.

\item  [(iv)]Another significant progress is due to   Carvalho, Pessil and Varandas \cite{cpv24}.  Since the  upper measure-theoretic metric mean dimensions of invariant measures, introduced in \cite{ycz22b,ycz22c,cpv24},  are difficult to compute its precise value except some   specific dynamical systems, they provided  a new strategy  of  formulating the measure-theoretic metric mean dimensions  of invariant measures,  and  established variational principles for metric mean dimension in terms of  a modified type of  Katok's   $\epsilon$-entropy  and  the local metric mean dimension function  of invariant measures \cite[Theorems C and E]{cpv24}.  Furthermore,  for the candidate $F(\mu,\epsilon)=\overline{h}_\mu^K(T,\epsilon,\delta)$,  an example \cite[Example 10.3]{cpv24} is exhibited  to  verify that   the order of   $\limsup_{\epsilon \to 0}$ and $\sup_{\mu \in M(X,T)}$  can not  be exchanged, which  gives a negative answer to  the previous  question.   
\end{itemize}

Based on the aforementioned facts, we can conclude that in general  the  order of $\limsup_{\epsilon \to 0}$ and $\sup_{\mu \in M(X,T)}$  of  the variational principles in Theorems \ref{thm 1.2} and \ref{thm 1.3} can not be  exchanged for  any TDS.  This is the reason why we pose the Question 3.  Notice from the Theorems \ref{thm 1.2} and \ref{thm 1.3} that  metric mean dimension can be expressed by different types of measure-theoretic $\epsilon$-entropies. This naturally leads to  another interesting question: 

 \emph{For  what  type of  special dynamical systems, the various type of variational principles  of metric mean  dimension, given  in Theorems \ref{thm 1.2} and \ref{thm 1.3}, are  unique for ergodic case?}
 
The uniqueness means that once  some potential conditions on TDSs are   confirmed for Question 3, then for such TDSs,  by Theorem \ref{thm 1.1},  all variational principles   are essentially reduced to the  following form:
\begin{align*}
	\over=\sup_{\mu \in E(X,T)}\limsup_{\epsilon \to 0}\frac{1}{\logf}\inf_{\diam P \leq \epsilon}\limits h_{\mu}(T,P).
\end{align*}

\subsection*{Acknowledgements}

~~~~We sincerely thank  the anonymous referee for  abundant valuable  comments  and  insightful suggestions that greatly improved the  quality of the paper. The second and third authors were supported by the
National Natural Science Foundation of China (Nos.12071222 and 11971236). The work was also funded by the Priority Academic Program Development of Jiangsu Higher Education Institutions.  We  also would like to express our gratitude to Tianyuan Mathematical Center in Southwest China(No.11826102), Sichuan University and Southwest Jiaotong University for their support and hospitality.

\end{document}